\documentclass[11pt]{amsart}
\usepackage{mathrsfs,amssymb, amsmath, array}
\usepackage{bbm}
\usepackage{verbatim}

\usepackage{graphicx}

\usepackage{url}
\usepackage[top = 3.75cm, bottom = 3.75cm, left = 3.7cm, right = 3.7cm]{geometry}
\usepackage{xcolor}
\usepackage{filecontents}
\usepackage{enumitem}
\usepackage{array}
\newcolumntype{L}{>{$}l<{$}}
\usepackage{float}
\usepackage[labelsep=period]{caption}

\usepackage{listings}
\usepackage[unicode]{hyperref}
\hypersetup{colorlinks,
	linkcolor=green!50!black,
	citecolor=blue!70!black,
	urlcolor=red!50!black
}

\newtheorem{theorem}{Theorem}[section]
\newtheorem{proposition}[theorem]{Proposition}
\newtheorem{lemma}[theorem]{Lemma}
\newtheorem{corollary}[theorem]{Corollary}
\newtheorem{definition}[theorem]{Definition}

\theoremstyle{remark}
\newtheorem*{remark}{Remark}

 \numberwithin{equation}{section}

\def \R {{\mathbb R}}

\def \N {{\mathbb N}}

\def \Z {{\mathbb Z}}
\def \Q {{\mathbb Q}}

\begin{document}

\title{On the Number of Lattice Points in Thin Sectors}
\author{Ezra Waxman}
\address{Department of Mathematics, University of Haifa, 3498838 Haifa, Israel.}
\email{ezrawaxman@gmail.com}
\author{Nadav Yesha}
\address{Department of Mathematics, University of Haifa, 3498838 Haifa, Israel.}
\email{nyesha@univ.haifa.ac.il}
\date{\today}
\subjclass[2020]{11P21, 11H06}

\keywords{Diophantine, lattice points, sectors}
\thanks{\textit{Data Availability Statement}: Data sharing is not applicable to this article as no datasets were generated or analysed during the current study.}
\thanks{\textit{Acknowledgements}: We thank Ze\'ev Rudnick and Andreas Strömbergsson for helpful discussions and comments, and the anonymous referee for useful comments and a careful read of the manuscript. This research was supported by the ISRAEL SCIENCE FOUNDATION (Grant No. 1881/20), and the first author was funded by a Zuckerman Post Doctoral Fellowship.}

\begin{abstract}
On the circle of radius $R$ centred at the origin, consider a ``thin'' sector about the fixed line $y = \alpha x$ with edges given by the lines $y = (\alpha \pm \epsilon) x$, where $\epsilon = \epsilon_R \rightarrow 0$ as $ R \to \infty $. We establish an asymptotic count for $S_{\alpha}(\epsilon,R)$, the number of integer lattice points lying in such a sector.  Our results depend both on the decay rate of $\epsilon$ and on the rationality/irrationality type of $\alpha$.  In particular, we demonstrate that if $\alpha$ is Diophantine, then $S_{\alpha}(\epsilon,R)$ is asymptotic to the area of the sector, so long as $\epsilon R^{t} \rightarrow \infty$ for some $ t<2 $.
\end{abstract}

\maketitle
\section{Introduction}
The \textit{Gauss circle problem} is the problem of determining how many integer lattice points lie inside a circle, centred at the origin, with radius $R\to\infty$.  This classical problem dates back to Gauss, who employed a simple geometric argument to show that the number of such lattice points is equal to the area of the circle, up to an error term of size $E\left(R\right) \leq 2\sqrt{2}\pi R.$  In 1906, Sierpi\'{n}ski \cite{Sierpinski} improved the bound on the error term to $E\left(R\right)=O(R^{2/3})$, and further incremental improvements have been subsequently made throughout the years.  The current state-of-the-art bound, due to Bourgain and Watt \cite{BourgainWatt}, is that $E\left(R\right)=O\left(R^{t+\varepsilon}\right)$ for any $\varepsilon>0$, where $t=517/824 \approx 0.6274$. It is famously
conjectured that $E\left(R\right)=O(R^{1/2+\varepsilon})$, for any
$\varepsilon>0$.

A natural related problem is to determine the number of lattice points $S(R)$ inside a \textit{sector} $\text{Sect}(R)$ of
a circle with radius $R\to\infty$. For sectors with \emph{fixed}
open angle, Gauss's argument can be easily extended to show that
\[
S\left(R\right)=\text{Area}\left(\text{Sect}\left(R\right)\right)+E\left(R\right),
\]
where $E\left(R\right)=O\left(R\right).$ Nowak  \cite{Nowak} (who, more generally, considered sectors in domains of the form $\left\{x^{n}+y^{n}\leq R^{n} : x,y\ge 0 \right\}$
for any $n\ge2)$ showed that the error term can be
improved 
when the slopes of the sector's two respective edges are either rational or irrational
of \emph{finite type} (see Definition \ref{irrationality measure def}).  Specifically, when both slopes are \emph{Diophantine}  (i.e. of type $\eta=1+\varepsilon$ for any $ \varepsilon >0 $), we have $ E(R)=O(R^{2/3-\delta}) $ for a certain (small) $ \delta>0 $.  Under a suitable assumption on the irrationality type of the edges' slopes, these results were further extended by
Kuba \cite{Kuba} to segments of even more general domains.
An additional closely related problem, dating back to the work of Hardy and Littlewood \cite{HardyLittlewood1,HardyLittlewood2},
concerns the number of lattice points in right-angled triangles.  An asymptotic
formula for this count $-$ which plays an important role in the proofs of \cite{Kuba} and \cite{Nowak} $-$ is obtained by applying Koksma's inequality together with standard discrepancy estimates (see, e.g. \cite[Theorem 3.2, p. 123 and Theorem 5.1, p. 143]{KuipersNiederreiter}).

In this paper we are  interested in counting the number of lattice points, $S_{\alpha}\left(\epsilon,R\right)$,
lying within a sector whose
open angle \emph{shrinks} as $R\to\infty$.  More explicitly, we consider a sector $\text{Sect}_{\alpha,\epsilon}\left(R\right)$ about the fixed line $y=\alpha x$ with edges given by the lines $y=\left(\alpha\pm\epsilon\right)x$, where now $\epsilon=\epsilon_{R}\to 0$ as $R\to\infty$. Our main goal is to establish an asymptotic formula for
$S_{\alpha}\left(\epsilon,R\right)$ rather than to optimize the relevant error term. In contrast to the case of fixed sectors, our results depend only on the rationality/irrationality type of $\alpha$, and \emph{not} on the rationality/irrationality type of $\alpha\pm\epsilon$, the slopes of the two edges. For this reason, the results of \cite{Kuba} and \cite{Nowak} are not applicable for our problem, and our argument proceeds in quite a different direction.

If $\epsilon \rightarrow 0$ at a rate slower than $1/R$, then upon applying a geometric argument similar to that used in the Gauss circle problem, we find that $S_{\alpha}\left(\epsilon,R\right) \sim \text{Area}\left(\text{Sect}_{\alpha,\epsilon}(R)\right)$ (see Theorem \ref{Slow Sectors} below). To obtain an asymptotic count for more quickly shrinking sectors, we must apply an alternative method.  First, we approximate $S_{\alpha}\left(\epsilon,R\right)$ by $\Delta_{\alpha}(\epsilon,R)$, the number of lattice point lying within a thin triangle whose two long edges lie along the lines $y=\left(\alpha\pm\epsilon\right)x$.  We then fix a rational number $p/q \in \Q$ that well-approximates $\alpha$, and compute $\Delta_{\alpha}(\epsilon,R)$ by summing the contributions from lattice points lying on a discrete collection of lines, each of which has rational slope $p/q$.

When $\alpha \in \Q$ is rational, we obtain an asymptotic for $S_{\alpha}\left(\epsilon,R\right)$, regardless of how fast our sectors shrink.  This is due to the fact that, in such a case, all the lattice points in $\text{Sect}_{\alpha,\epsilon}(R)$ lie precisely on the line $ y=\alpha x$ once $\text{Sect}_{\alpha,\epsilon}(R)$ is sufficiently thin.  If $ \alpha $ is irrational of finite type $\eta$, we obtain an asymptotic for $S_{\alpha}\left(\epsilon,R\right)$ under the assumption that $\epsilon$ decays at a rate slower than $ 1/R^{1+1/\eta}$ (Theorem \ref{main theorem} below).  Specifically, when $\alpha \in \R$ is Diophantine, we obtain an asymptotic for $S_{\alpha}\left(\epsilon,R\right)$ under the assumption that  $\epsilon \rightarrow 0$ at a rate slower than $ 1/R^t$ for some $ t<2 $.  

The behaviour of lattice points in even faster shrinking sectors about irrational slopes is a more subtle question.  If $\epsilon$ decays at a rate $1/R^{1+1/\eta}$ or faster, the above method fails to produce an asymptotic count for $S_{\alpha}\left(\epsilon,R\right)$.
However, if $\epsilon$ shrinks \textit{sufficiently} quickly, then the count once again becomes much simpler.  Specifically, when $\epsilon$ decays faster than $1/R^{1+\eta}$, we may apply an elementary argument to prove that for sufficiently large $R$, $\text{Sect}_{\alpha,\epsilon}(R)$ contains no lattice points whatsoever (Proposition \ref{quickly shrinking}).  A related question concerns the distribution of lattice points in a \textit{randomly} chosen sector of width $\epsilon \asymp 1/R^2$.  This interesting question has been addressed by Marklof and Strömbergsson \cite{MarklofStrombergsson}, who successfully applied tools from homogeneous dynamics to prove the existence of a (non-Poissonian) limiting distribution for the number of lattice points in such sectors.

\subsection{Notation}\label{sec:main_results}

Fix $\alpha \in \R$, and consider the interval $I_{\epsilon}(\alpha):= \left(\alpha - \epsilon,\alpha+\epsilon \right)$, for some $\epsilon > 0$.
Let
\begin{align*}
\textnormal{Sect}_{\alpha, \epsilon}(R)&:=\{(x,y) \in \R_{>0}\times \R: x^{2}+y^{2} \leq R^{2}, y/x \in I_{\epsilon}(\alpha)\}
\end{align*}
denote the sector of radius $R$ with edges given by the lines $y=\left(\alpha\pm\epsilon\right)x$, which has an open angle of size
\[\theta :=\tan^{-1}(\alpha+\epsilon)-\tan^{-1}(\alpha-\epsilon).\] 
In what follows, we view $\epsilon = \epsilon_R $ as a function of $R$. Our main interest will be in \textit{thin} sectors, i.e. when $ \epsilon \to 0 $ as $ R \to \infty $.
Taylor expanding about $\alpha$, we find that as $ \epsilon \to 0 $, the area of $\textnormal{Sect}_{\alpha, \epsilon}(R)$ is equal to
\begin{align*}
\textnormal{Area}(\textnormal{Sect}_{\alpha, \epsilon}(R)) &= \frac{R^2}{2}\left(\tan^{-1}(\alpha + \epsilon)-\tan^{-1}(\alpha - \epsilon)\right)= \frac{\epsilon R^{2}}{1+\alpha^{2}} + O\left(\epsilon^{3}R^{2} \right).
\end{align*}

Let
\begin{align*}
S_{\alpha}(\epsilon,R)&:= \#\{\Z^{2} \cap \textnormal{Sect}_{\alpha,\epsilon}(R) \}\\
&\phantom{:}= \#\{(m,n) \in \mathbb{Z}_{>0} \times \Z: m^{2}+n^{2} \leq R^{2}, |n - \alpha m| < m\epsilon\}
\end{align*}
count the number of integer lattice points in $\textnormal{Sect}_{\alpha,\epsilon}(R)$.

We are interested in the value of $S_{\alpha}(\epsilon,R)$ in the limit as $R \rightarrow \infty$.  For example, we may consider the case $\epsilon:= R^{-\lambda}$ for some fixed $\lambda \geq 0$.  We then classify our sectors based on the decay rate of $\epsilon$.

\begin{remark}
Our results may be easily extended to more general sectors about the line $ y=\alpha x$. In particular, we note that Theorem \ref{main theorem} continues to hold when counting lattice points in any sector of the form 
\[\textnormal{Sect}_{\alpha,\epsilon_{1},\epsilon_{2}}(R):=\{(x,y) \in \R_{>0} \times \R: x^2+y^2 \leq R^2, y/x\in(\alpha-\epsilon_1, \alpha+\epsilon_2)\},\]
where, say, $\epsilon_{1} \asymp \epsilon_{2} \asymp \epsilon$.  Consequently, one may alternatively consider a sector centred about the angle $\Phi:=\tan^{-1}(\alpha)$ with radius $R$  and open angle $\theta \asymp \epsilon$; and express the resulting lattice point count in terms of the properties of $\tan{\Phi}$ and the decay rate of $\theta$ without any alterations to Theorem \ref{main theorem}.  Nonetheless, we have chosen to formulate our results in terms of slopes, rather than angles, in order to simplify the exposition, and because our analysis naturally depends upon the Diophantine properties of the slope $\alpha$.
\end{remark}
	
\subsection{``Slowly" Shrinking Sectors}
Suppose first that $ \epsilon $ is either fixed or decays slower than $ 1/R $, in the sense that $\epsilon R \rightarrow \infty$ in the limit as $R \rightarrow \infty$ (e.g. $0 \leq \lambda < 1$).  Upon refining the elementary geometric argument of the $ O(R) $ bound for the error term in the Gauss circle problem, we obtain the following result, which yields an asymptotic count for the number of lattice points in such slowly shrinking sectors:

\begin{theorem}\label{Slow Sectors}
Fix $\alpha \in \mathbb{R}$, and assume that $ \epsilon R \to \infty $ as $ R \to \infty$. Then
\begin{equation}\label{asymptotic slow shrinking}
S_{\alpha}(\epsilon,R)= \textnormal{Area}(\textnormal{Sect}_{\alpha, \epsilon}(R))+O\left(R\right).\end{equation}
\end{theorem}

\subsection{``Quickly" Shrinking Sectors}
In our investigation of more quickly shrinking sectors, our results depend heavily upon the rationality/irrationality type of $\alpha$, defined as follows:

\begin{definition}\label{irrationality measure def}
	We say that an irrational $\alpha \in \R$ is of finite \textbf{type} $\eta$, if there exists a constant $c=c(\alpha)>0$ such that
	\[\bigg |\alpha - \frac{p}{q}\bigg |>  \frac{c}{q^{1+\eta}}\]
	for all integers pairs $(p,q) \in \Z \times \Z_{>0}$.
\end{definition}

Note that for irrational $\alpha \in \R$ of type $\eta$ we necessarily have $\eta \geq 1$
by Dirichlet's theorem.  We say that $\alpha$ is \textbf{Diophantine} if $\alpha \in \R$ is irrational of type $\eta = 1+\varepsilon$ for every $\varepsilon > 0$.
It is well-known that almost all $\alpha \in \R$ are Diophantine (Khinchin's theorem), and every algebraic number is Diophantine (Roth's theorem).

\subsubsection{Irrational Slopes}
For irrational $\alpha \in \R$, our main result is as follows:

\begin{theorem}\label{main theorem}
Let $\alpha \in \R$ be irrational of finite type $\eta$, and assume that $ \epsilon \to 0 $, as well as that $ \epsilon R^{1+1/\eta} \to \infty$ as $ R\to \infty $. Then
\[S_{\alpha}(\epsilon,R) = \textnormal{Area}(\textnormal{Sect}_{\alpha, \epsilon}(R)) +O\left(\epsilon^{\frac{1}{1+\eta}} R+R^{2}\epsilon^{2}\right),\]
in the limit as $R \rightarrow \infty$.
\end{theorem}
The conditions $ \epsilon \to 0 $ and  $ \epsilon R^{1+1/\eta} \to \infty$ (e.g. $ 0 < \lambda < 1+1/\eta)$ consequently guarantee the asymptotic
\begin{equation}\label{eq:main_asymptotic}
S_{\alpha}(\epsilon,R) \sim \textnormal{Area}(\textnormal{Sect}_{\alpha, \epsilon}(R)).
\end{equation}
In particular, if $\alpha$ is Diophantine, then \eqref{eq:main_asymptotic} holds whenever $ \epsilon \to 0 $ and $\epsilon R^{t} \rightarrow \infty$ for some $ t<2 $ (e.g. $0 < \lambda < 2$).  Note furthermore that $\textnormal{Area}(\textnormal{Sect}_{\alpha, \epsilon}(R))$ grows if and only if $\epsilon R^{2} \rightarrow \infty$, and thus our results in such a case are essentially optimal (and ``strictly" so whenever $\alpha \in \R$ is a badly approximable irrational, i.e. irrational of type $\eta = 1$).\\
\\
Note that Theorem \ref{main theorem} gives a better error term than \eqref{asymptotic slow shrinking} whenever $ \epsilon = o(R^{-1/2}) $ and $ \epsilon R \to \infty $ (e.g. $ 1/2<\lambda <1 $). Upon comparing the error terms in Theorem \ref{Slow Sectors} and Theorem \ref{main theorem} we obtain the following corollary:

\begin{corollary}\label{main corollary}
Let $\alpha \in \R$ be irrational of finite type $\eta$, and let $ \epsilon= R^{-\lambda} $.  Then in the limit as $R \rightarrow \infty$,

\begin{equation}
S_{\alpha}(\epsilon,R) = \left\{
\begin{array}{l l l}
\textnormal{Area}(\textnormal{Sect}_{\alpha, \epsilon}(R))+O(R), &  0 \leq \lambda < \frac{1}{2} \\
\textnormal{Area}(\textnormal{Sect}_{\alpha, \epsilon}(R))+O\left(R^{2-2 \lambda}\right), & \frac{1}{2} \leq \lambda < \frac{1+\eta}{1+2 \eta}\\
\textnormal{Area}(\textnormal{Sect}_{\alpha, \epsilon}(R))+O\left(R^{1-\frac{\lambda}{1+\eta}}\right), &   \frac{1+\eta}{1+2 \eta} \leq \lambda < 1+\frac{1}{\eta}.
\end{array} \right.
\end{equation}
\end{corollary}

 In particular, when $\alpha \in \R$ is Diophantine, Corollary \ref{main corollary} yields

\begin{equation}
S_{\alpha}(\epsilon,R) = \left\{
\begin{array}{l l l}
\textnormal{Area}(\textnormal{Sect}_{\alpha, \epsilon}(R))+O(R), & 0 \leq \lambda <\frac{1}{2} \\
\textnormal{Area}(\textnormal{Sect}_{\alpha, \epsilon}(R))+O\left(R^{2-2\lambda}\right), & \frac{1}{2} \leq \lambda <\frac{2}{3}\\
\textnormal{Area}(\textnormal{Sect}_{\alpha, \epsilon}(R))+O_\delta\left(R^{1-\lambda/2+\delta}\right), &  \frac{2}{3} \leq \lambda < 2.
\end{array} \right.
\end{equation}

\subsubsection{Rational Slopes}
For rational $\alpha \in \Q$, we obtain the following result:

\begin{theorem}\label{rational case thm}
Fix $\alpha = p/q \in \Q$, where $q>0$ and $(p,q)=1$.  Then in the limit as $R \rightarrow \infty$, we have
\begin{align*}
S_{\alpha}(\epsilon,R)=
\frac{\epsilon q^{2}R^{2}}{p^{2}+q^{2}} +\frac{1}{q^{2}\epsilon}\left\{\frac{\epsilon q^2 R}{\sqrt{p^{2}+q^{2}}}\right\}\left(1-\left\{\frac{\epsilon q^2 R}{\sqrt{p^{2}+q^{2}}}\right\}\right)+O\left(1+(R\epsilon)^{2}\right),
\end{align*}
where $\{x\}:=x - \lfloor x \rfloor$ denotes the fractional part of $x$.
\end{theorem}

When $\epsilon = o(R^{-1})$ (e.g. $\lambda > 1$), Theorem \ref{rational case thm} simplifies to

\[S_{\alpha}(\epsilon,R) =
\frac{R}{\sqrt{p^{2}+q^{2}}}+O(1).\]

In this case, $S_{\alpha}(\epsilon,R)$ is no longer asymptotic to $\textnormal{Area}(\textnormal{Sect}_{\alpha, \epsilon}(R))$, and the only points that contribute to $S_{\alpha}(\epsilon,R)$ are those which lie precisely on the line $y = \alpha x$.

When $\epsilon\to 0$ and $\epsilon  R \rightarrow \infty$ (e.g. $0 < \lambda < 1$), Theorem \ref{rational case thm} yields

\begin{equation}\label{rational slow shrinking}
S_{\alpha}(\epsilon,R) =\textnormal{Area}(\textnormal{Sect}_{\alpha, \epsilon}(R))+\beta/\epsilon + O(\epsilon^{2}R^{2}),
\end{equation}

where

\[\beta:=\frac{1}{q^{2}}\left\{\frac{\epsilon q^2  R}{\sqrt{p^{2}+q^{2}}}\right\}\left(1-\left\{\frac{\epsilon q^2  R}{\sqrt{p^{2}+q^{2}}}\right\}\right)\]
is a bounded function of $R$.  In particular, as in the case of irrational slopes, 
if $\epsilon = o(R^{-1/2})$ and $\epsilon R \rightarrow \infty$ (e.g. $1/2 < \lambda < 1$), then (\ref{rational slow shrinking}) yields a more precise count than (\ref{asymptotic slow shrinking}). The following corollary summarizes the above analysis in the case $ \epsilon = R^{-\lambda} $:

\begin{corollary}\label{rational corollary}
	Let $\alpha = p/q \in \Q$, where $q>0$ and $(p,q)=1$, and let $ \epsilon= R^{-\lambda} $.  Then in the limit as $R \rightarrow \infty$, we have
	
	\begin{equation}
		S_{\alpha}(\epsilon,R) = \left\{
		\begin{array}{l l l l}
			\textnormal{Area}(\textnormal{Sect}_{\alpha, \epsilon}(R))+O(R), & 0 \leq \lambda \leq \frac{1}{2} \\
			\textnormal{Area}(\textnormal{Sect}_{\alpha, \epsilon}(R))+O\left(R^{2-2\lambda}\right), & \frac{1}{2} <\lambda \leq \frac{2}{3}\\
			\textnormal{Area}(\textnormal{Sect}_{\alpha, \epsilon}(R))+\beta R^\lambda +O\left(R^{2-2\lambda}\right), & \frac{2}{3} < \lambda < 1 \\
			\frac{R}{\sqrt{p^{2}+q^{2}}}+O(1), &   1<  \lambda .
		\end{array} \right.
	\end{equation}
\end{corollary}
Finally, we consider the case $\epsilon \asymp R^{-1}$ (e.g. $\epsilon R = c$, for some $c \in \R_{> 0})$.  Then Theorem \ref{rational case thm} yields
\[S_{\alpha}(\epsilon,R) =
\gamma  R+O(1),\]
where
\[\gamma:=\frac{\epsilon q^{2}R }{p^{2}+q^{2}} +\frac{1}{\epsilon q^{2} R}\left\{\frac{\epsilon q^2 R }{\sqrt{p^{2}+q^{2}}}\right\}\left(1-\left\{\frac{\epsilon q^2 R }{\sqrt{p^{2}+q^{2}}}\right\}\right).\]
In particular, whenever $\epsilon < \frac{\sqrt{p^{2}+q^{2}}}{q^{2} R}$, the only points which contribute to $S_{\alpha}(\epsilon,R)$ are those which lie precisely on the line $y = \alpha x$, and
we find that
\[\gamma = \frac{1}{\sqrt{p^{2}+q^{2}}}.\]
We moreover note that $S_{\alpha}(\epsilon,R)$ is asymptotic to $\textnormal{Area}(\textnormal{Sect}_{\alpha, \epsilon}(R))$
if and only if $\gamma = \frac{\epsilon q^{2}R}{p^{2}+q^{2}}$, i.e. if and only if $ \epsilon $
is an integer multiple of $\frac{\sqrt{p^{2}+q^{2}}}{q^{2} R}$.

\subsection{``Very Quickly" Shrinking Sectors}
While in the range $R^{-1-\eta} \ll \epsilon  \ll R^{-(1+1/\eta)}$ we are unable to obtain an asymptotic formula for $S_{\alpha}(\epsilon,R)$, for sectors that shrink even more quickly the situation becomes rather trivial.  Specifically, whenever $\epsilon = o(R^{-1-\eta}) $ (e.g. $\lambda >1+ \eta$), we show that $S_{\alpha}(\epsilon,R) =0$ for sufficiently large $R$:
\begin{proposition}\label{quickly shrinking}
Let $\alpha \in \R$ be irrational of finite type $\eta$, and suppose that  $\epsilon = o(R^{-1-\eta}) $.  Then there exists $R_{0} > 0$ such that for all $R > R_{0}$,
\[S_{\alpha}(\epsilon,R) = 0.\]
\end{proposition}
In particular, if $\alpha$ is a Diophantine irrational, then for sufficiently large $R$, $S_{\alpha}(\epsilon,R) = 0$ whenever $\epsilon = o(R^{-t})$ for some $ t>2 $ (e.g. $\lambda > 2$).

\subsection{Structure of Paper} The remainder of this paper is structured as follows.
In Section \ref{sec:slow_sectors} we apply a simple geometric argument to compute $S_{\alpha}(\epsilon,R)$ in the case that $\epsilon \rightarrow 0$ at a rate slower than $1/R$. In Section \ref{sec:triangle_approx} we approximate $S_{\alpha}(\epsilon,R)$ by 
$\Delta_{\alpha}(\epsilon,R)$, i.e. by the number of lattice points in a triangle whose two long edges lie along the lines $y=\left(\alpha\pm\epsilon\right)x$. In Section \ref{sec:irrational_slopes} we then apply this approximation to compute $S_{\alpha}(\epsilon,R)$ when $\alpha \in \R$ is irrational of finite type; and in Section \ref{sec:rational_slopes} we address the case when $\alpha \in \Q$ is rational.  Finally, in Section \ref{sec:very_quick_sectors}, we address the case in which $\text{Sect}_{\alpha,\epsilon}(R)$ shrinks ``very quickly", i.e. when $\epsilon \rightarrow 0$ at a rate faster than $1/R^{1+\eta}$.

\section{Lattice Points in Slowly Shrinking Sectors}\label{sec:slow_sectors}
In this section we provide a proof of Theorem \ref{Slow Sectors}, namely a count for $S_{\alpha}(\epsilon,R)$ when $\epsilon R \rightarrow \infty$ as $R \rightarrow \infty$.  The proof is an easy adaptation of the elementary geometric argument applied in the classical Gauss circle problem.  As evidenced by the proof, this argument remains valid for slowly shrinking sectors.

\begin{proof}[Proof of Theorem \ref{Slow Sectors}]
For each $z \in \Z^{2} \cap \textnormal{Sect}_{\alpha,\epsilon}(R)$, let $\square_{z}$ denote a square-box of unit area, centred at the point $z$.  Then
\[S_{\alpha}(\epsilon,R) = \textnormal{Area}\left(\bigcup_{z \in \Z^{2} \cap \textnormal{Sect}_{\alpha,\epsilon}(R)}\square_{z}\right),\]
i.e. $S_{\alpha}(\epsilon,R)$ is equal to the area formed by the union of such boxes.  Note, moreover, that if $w \in \square_{z}$ for some $z \in \Z^{2} \cap \textnormal{Sect}_{\alpha,\epsilon}(R)$, then
\[\textnormal{dist}(w,\textnormal{Sect}_{\alpha,\epsilon}(R)) \leq \sqrt{2}/2,\]
i.e. the distance between $w$ and $\textnormal{Sect}_{\alpha,\epsilon}(R)$ is bounded by $\sqrt{2}/2$.  We therefore define a wider sector, $\textnormal{Sect}^{+}_{\alpha,\epsilon}(R')$, with the same open angle and direction as $\textnormal{Sect}_{\alpha,\epsilon}(R)$, but extended by a distance of $\sqrt{2}/2$ on all sides, so that
\begin{equation*}
\bigcup_{z \in \Z^{2} \cap \textnormal{Sect}_{\alpha,\epsilon}(R)}\square_{z} \subseteq \textnormal{Sect}^{+}_{\alpha,\epsilon}(R').
\end{equation*}

To construct $\textnormal{Sect}^{+}_{\alpha,\epsilon}(R')$ explicitly, we expand $\textnormal{Sect}_{\alpha,\epsilon}(R)$ by drawing parallel lines distanced $d=\sqrt{2}/2$ away from each of its two respective straight edges.  Let $x$ denote the distance between their point of intersection and the origin.  Note that
\[x \cdot \sin \frac{\theta}{2} = \frac{\sqrt{2}}{2},\]
from which we obtain
\[x = \frac{\sqrt{2}}{2 \cdot \sin \frac{\theta}{2} } = \frac{\sqrt{2}}{\theta+O(\theta^{3})} =\frac{\sqrt{2}}{\theta}+O(\theta).\]
We therefore set the radius of our desired sector, $\textnormal{Sect}^{+}_{\alpha,\epsilon}(R')$, to be equal to
\[R' = R + \frac{\sqrt{2}}{2 \cdot \sin \frac{\theta}{2} }+ \frac{\sqrt{2}}{2} = R + \frac{\sqrt{2}}{\theta}+ O(1),\]
which yields
\begin{align*}
\textnormal{Area}\left(\textnormal{Sect}^{+}_{\alpha,\epsilon}(R')\right) &= \frac{\theta}{2}  \cdot (R')^{2} = \frac{\theta}{2}  \cdot \left(R + \frac{\sqrt{2}}{\theta}+ O(1)\right)^{2}\\
&=  \textnormal{Area}(\textnormal{Sect}_{\alpha, \epsilon}(R))+O\left(R\right),
\end{align*}
upon noting that $ \theta \asymp \epsilon $, so that $ \theta^{-1}=o(R) $. Thus
\begin{align}\label{theta upper bound}
S_{\alpha}(\epsilon,R) &= \textnormal{Area}\left(\bigcup_{z \in \Z^{2} \cap \textnormal{Sect}_{\alpha,\epsilon}(R)}\square_{z}\right) \leq \textnormal{Area}\left(\textnormal{Sect}^{+}_{\alpha,\epsilon}(R')\right) \nonumber \\
&= \textnormal{Area}(\textnormal{Sect}_{\alpha, \epsilon}(R))+O\left(R\right).
\end{align}
To obtain a lower bound for $S_{\alpha}(\epsilon,R)$, we similarly construct a sector, denoted by $\textnormal{Sect}^{-}_{\alpha,\epsilon}(R'')$, with the same open angle and direction as $\textnormal{Sect}_{\alpha,\epsilon}(R)$, but now \textit{shrunk} by a distance of $\sqrt{2}/2$ on all sides, of radius \[ R'' = R - \frac{\sqrt{2}}{2 \cdot \sin \frac{\theta}{2} }- \frac{\sqrt{2}}{2},\] which we note is clearly possible since $ \theta^{-1} = o(R) $.   Any point $w \in \textnormal{Sect}^{-}_{\alpha,\epsilon}(R'')$ is within a distance of at most $\sqrt{2}/2$ from some lattice point $z$, which, by construction, must lie in $\textnormal{Sect}_{\alpha,\epsilon}(R)$. It follows that
\[\textnormal{Sect}^{-}_{\alpha,\epsilon}(R'') \subseteq \left(\bigcup_{z \in \Z^{2} \cap \textnormal{Sect}_{\alpha,\epsilon}(R)}\square_{z}\right).\]
Using a similar analysis to that above, we find that
\[\textnormal{Area} \left(\textnormal{Sect}^{-}_{\alpha,\epsilon}(R'')\right) = \frac{\theta}{2}  \cdot R^{2}+O\left(R\right),\]
and therefore
\begin{equation}\label{theta lower bound}
S_{\alpha}(\epsilon,R) \geq \textnormal{Area} \left(\textnormal{Sect}^{-}_{\alpha,\epsilon}(R'')\right) =  \textnormal{Area}(\textnormal{Sect}_{\alpha, \epsilon}(R))+O\left(R\right).
\end{equation}
Combining (\ref{theta upper bound}) and (\ref{theta lower bound}) we conclude that

\begin{equation}
S_{\alpha}(\epsilon,R) =  \textnormal{Area}(\textnormal{Sect}_{\alpha, \epsilon}(R))+O\left(R\right),
\end{equation}
as desired.
\end{proof}
\section{Approximating Sectors by Triangles}\label{sec:triangle_approx}
In this section we approximate $S_{\alpha}(\epsilon,R)$ by considering lattice points in a triangle, namely the summation
\[\Delta_{\alpha}(\epsilon,R):=\sum _{1 \leq m \leq \frac{R}{\sqrt{1+\alpha^{2}}}}\#\{n \in \Z: m(\alpha-\epsilon)< n < m(\alpha+\epsilon)\}\]
We have the following lemma:

\begin{lemma}\label{S and Delta}
Assume that $\epsilon \to 0$.  Then
\[S_{\alpha}(\epsilon,R) = \Delta_{\alpha}(\epsilon,R)+O\left( 1+(R\epsilon)^{2}\right)\]
In particular, if $\epsilon = O(R^{-1})$, then
\[S_{\alpha}(\epsilon,R) = \Delta_{\alpha}(\epsilon,R)+O(1).\]
\end{lemma}

\begin{proof}
Assume $\alpha > 0$, as the proof for the cases $\alpha = 0$ and $\alpha < 0$ follow similarly.  Suppose $(m,n) \in S_{\alpha}(\epsilon,R)$.  Then $m^{2}+n^{2}\leq R^{2}$ and $n > m(\alpha-\epsilon) > 0$ (which holds for sufficiently small $ \epsilon $) together imply
\begin{align*}
 m^{2}\left(1+(\alpha-\epsilon)^{2})\right)\leq  R^{2},
\end{align*}
i.e. that
\[m \leq \frac{R}{\sqrt{1+(\alpha-\epsilon)^{2}}}.\]
We may therefore write
\[S_{\alpha}(\epsilon,R) = S^{1}_{\alpha}(\epsilon,R)-S^{2}_{\alpha}(\epsilon,R),\]
where
\[S^{1}_{\alpha}(\epsilon,R) := \sum_{1 \leq m \leq \frac{R}{\sqrt{1+(\alpha-\epsilon)^{2}}}}\#\{n: m(\alpha-\epsilon)< n < m(\alpha+\epsilon)\}\]
and

\[S^{2}_{\alpha}(\epsilon,R) := \sum_{1 \leq m \leq \frac{R}{\sqrt{1+(\alpha-\epsilon)^{2}}}}\#\{n: m(\alpha-\epsilon)< n < m(\alpha+\epsilon), m^{2}+n^{2}> R^{2}\}.\]
Let us first estimate the size of $S^{2}_{\alpha}(\epsilon,R)$.  Note that if $m^{2}+n^{2}> R^{2}$ and $m(\alpha-\epsilon)< n < m(\alpha+\epsilon)$, then $m^{2}(1+(\alpha+\epsilon)^{2})>R^{2}$, and therefore $m > R/\sqrt{1+(\alpha+\epsilon)^{2}}$.  Moreover, since the length of the interval $(m(\alpha-\epsilon),m(\alpha+\epsilon))$ is $2m \epsilon \leq 2R \epsilon$, we find that, for any $m \in \N$, there exist at most $O(1+R\epsilon)$ integers $n \in \Z$ such that
\[m(\alpha-\epsilon) < n < m(\alpha+\epsilon).\]
Thus
\begin{align*}
S^{2}_{\alpha}(\epsilon,R) &\ll \sum_{\frac{R}{\sqrt{1+(\alpha+\epsilon)^{2}}} < m \leq \frac{R}{\sqrt{1+(\alpha-\epsilon)^{2}}}}(1+R\epsilon)\\
&\leq (1+R\epsilon)\left(1+
\frac{R}{\sqrt{1+(\alpha-\epsilon)^{2}}}-\frac{R}{\sqrt{1+(\alpha+\epsilon)^{2}}}\right).
\end{align*}
Note furthermore that
\begin{align*}
\sqrt{1+(\alpha \pm \epsilon)^{2}} &= \sqrt{1+\alpha^2}(1+ O(\epsilon)).
\end{align*}
It follows that

\begin{align*}
\frac{R}{\sqrt{1+(\alpha-\epsilon)^{2}}}-\frac{R}{\sqrt{1+(\alpha+\epsilon)^{2}}}&=\frac{R}{\sqrt{1+\alpha^{2}}}\left(1+O(\epsilon)\right)-\frac{R}{\sqrt{1+\alpha^{2}}}\left(1+O(\epsilon)\right)\\
&=O(R \epsilon).
\end{align*}
Hence
\[S^{2}_{\alpha}(\epsilon,R) \ll (1+R\epsilon)^{2} \ll 1+(R\epsilon)^{2},\]
from which we obtain that

\[S_{\alpha}(\epsilon,R) = S^{1}_{\alpha}(\epsilon,R)+O\left(1+(R\epsilon)^{2}\right).\]
Next, we wish to show that
\[S^{1}_{\alpha}(\epsilon,R)=\Delta_{\alpha}(\epsilon,R)+O\left(1+(R\epsilon)^{2}\right).\]
Indeed, note that
\begin{equation}\label{diff S and Delta}
S^{1}_{\alpha}(\epsilon,R) - \Delta_{\alpha}(\epsilon,R)= \sum_{\frac{R}{\sqrt{1+\alpha^{2}}}< m \leq \frac{R}{\sqrt{1+(\alpha-\epsilon)^{2}}}}\#\{n: m(\alpha-\epsilon)< n < m(\alpha+\epsilon))\},
\end{equation}
and that each summand in (\ref{diff S and Delta}) is O$(1+R\epsilon)$. It follows that
\begin{align*}
S^{1}_{\alpha}(\epsilon,R) - \Delta_{\alpha}(\epsilon,R) &\ll (1+R\epsilon)\cdot \left(1+\frac{R}{\sqrt{1+(\alpha-\epsilon)^{2}}}-\frac{R}{\sqrt{1+\alpha^{2}}}\right)\\
&\ll (1+R\epsilon)^{2} \ll 1+(R\epsilon)^{2},
\end{align*}
as desired.
\end{proof}

\section{Sectors about Irrational Slopes}\label{sec:irrational_slopes}
In this section we provide a proof of Theorem \ref{main theorem}, namely a count for $S_{\alpha}(\epsilon,R)$ when $\alpha \in \R$ is irrational of finite type.\\
\\
Let $\alpha \in \R$ be irrational.  For any rational $p/q \in \Q$, we define $\delta:=\alpha - p/q$.  For the purposes of this proof, we will moreover assume that $|\delta|< \epsilon/2$, which, in particular, implies that $ \delta-\epsilon<0 $ and $ \epsilon+\delta>0 $. We then write
\begin{align*}
\Delta_{\alpha}(\epsilon,R)&=\{(m,n) \in \Z^{2}:|n/m - \alpha|< \epsilon, \hspace{2mm} 1 \leq m \leq R/\sqrt{1+\alpha^{2}}\}\\
&=\{(m,n) \in \Z^{2}:-\epsilon + \delta <n/m - p/q < \epsilon + \delta, \hspace{2mm} 1 \leq m \leq R/\sqrt{1+\alpha^{2}}\}\\
&=\{(m,n) \in \Z^{2}: mq(\delta-\epsilon) <nq - mp < (\epsilon + \delta)mq, \hspace{2mm} 1 \leq m \leq R/\sqrt{1+\alpha^{2}}\}.
\end{align*}
Let $d= n q -m p$, so that 
\[(\delta-\epsilon)m q < d < (\epsilon + \delta)mq.\]
Together with the conditions on $m$, this implies that
\[\frac{(\delta-\epsilon)q R}{\sqrt{1+\alpha^{2}}} \leq d \leq \frac{(\epsilon + \delta)q R}{\sqrt{1+\alpha^{2}}}.\]
When $d > 0$, the condition on $m$ is equivalent to
\[\frac{d}{(\delta+\epsilon)q} < m \leq \frac{R}{\sqrt{1+\alpha^{2}}},\]
while when $d < 0$, the condition is then
\[\frac{d}{(\delta - \epsilon)q}< m \leq \frac{R}{\sqrt{1+\alpha^{2}}}.\]
Partitioning with respect to $d$, we then write

\begin{equation}\label{eq:Delta partition}
\Delta_{\alpha}(\epsilon,R) = \Delta_{\alpha}^{+}(\epsilon,R)+\Delta_{\alpha}^{0}(\epsilon,R)+\Delta_{\alpha}^{-}(\epsilon,R),
\end{equation}
with
\begin{align*}
\Delta_{\alpha}^{+}(\epsilon,R)&:=\sum_{0 < d \leq \frac{(\epsilon +\delta)Rq}{\sqrt{1+\alpha^{2}}}}\sum_{\substack{\frac{d}{(\epsilon + \delta)q}< m \leq \frac{R}{\sqrt{1+\alpha^{2}}} \\ m \equiv -d \bar{p} \; (q)}} 1\\
\Delta^{-}_{\alpha}(\epsilon,R)&:= \sum_{ \frac{(\delta -\epsilon)Rq}{\sqrt{1+\alpha^{2}}} \leq d < 0} \sum_{\substack{\frac{d}{q( \delta-\epsilon)}< m \leq \frac{R}{\sqrt{1+\alpha^{2}}} \\ m \equiv -d\bar{p} \; (q)}} 1,\\
\Delta^{0}_{\alpha}(\epsilon,R)&:= \sum_{\substack{1 \leq m \leq \frac{R}{\sqrt{1+\alpha^{2}}} \\ m \equiv 0 \; (q)}}1,
\end{align*}
where $ \bar{p} $ denotes the inverse of $ p $ modulo $ q $.  Upon recalling that
\[\sum_{0 < d\leq x}d = \frac{\lfloor x \rfloor \left( \lfloor x \rfloor +1 \right) }{2} = \frac{1}{2}\left(x+O(1)\right)^2,\]
we see that
\begin{align*}
\Delta_{\alpha}^{+}(\epsilon,R)&=\sum_{0 < d \leq \frac{(\epsilon +\delta)Rq}{\sqrt{1+\alpha^{2}}}}\left(\frac{1}{q}\left(\frac{R}{\sqrt{1+\alpha^{2}}}-\frac{d}{(\epsilon + \delta)q}\right)+O(1)\right)\\
&=\frac{R}{q\sqrt{1+\alpha^{2}}}\left( \frac{(\epsilon + \delta)Rq}{\sqrt{1+\alpha^2}}+O(1)\right)\\
&\phantom{=}-\frac{1}{2(\epsilon+\delta)q^{2}}\left( \frac{(\epsilon +\delta) R q}{\sqrt{1+\alpha^{2}}}+O\left(1\right) \right)^2 +O(\epsilon q R )\\
&=\frac{1}{2}\frac{R^{2}(\epsilon+\delta)}{1+\alpha^{2}}+O\left(\frac{R}{q}+\frac{1}{\epsilon q^2}+\epsilon q R\right).
\end{align*}
Similarly, we compute
\begin{align*}
\Delta^{-}_{\alpha}(\epsilon,R)
&= \sum_{\frac{(\delta -\epsilon)Rq}{\sqrt{1+\alpha^{2}}} \leq d < 0}\left(\frac{1}{q}\left(\frac{R}{\sqrt{1+\alpha^{2}}}-\frac{d}{(\delta - \epsilon)q}\right)+O(1)\right)\\
&= \sum_{0 < d \leq \frac{(\epsilon - \delta)Rq}{\sqrt{1+\alpha^{2}}}}\left(\frac{1}{q}\left(\frac{R}{\sqrt{1+\alpha^{2}}}-\frac{d}{(\epsilon - \delta)q}\right)+O(1)\right)\\
&=\frac{1}{2}\frac{R^{2}(\epsilon-\delta)}{1+\alpha^{2}}+O\left(\frac{R}{q}+\frac{1}{\epsilon q^2}+\epsilon q R\right).
\end{align*}
Finally, we note that
\begin{equation}\label{eq:0 term contribution}
\Delta_{\alpha}^{0}(\epsilon, R)= \frac{R}{q\sqrt{1+\alpha^{2}}}+O(1).
\end{equation}
It then follows from \eqref{eq:Delta partition} that
\begin{equation}\label{eq:Delta in q}
\Delta_{\alpha}(\epsilon, R) = \frac{\epsilon R^{2}}{(1+\alpha^{2})}+O\left(\frac{R}{q}+\frac{1}{\epsilon q^2}+\epsilon q R+1\right).
\end{equation}

\subsection{Choosing an Appropriate Convergent}
Suppose $\alpha \in \R$ is irrational of finite type $\eta$, and let $\{p_{i}/q_{i}\}_{i=1}^{\infty}$ denote the sequence of convergents to the continued fraction of $\alpha$.  Upon choosing an appropriate pair $\{p_{i}/q_{i}\}$, we are able to proceed with a proof of Theorem \ref{main theorem}:

\begin{proof}[Proof of Theorem \ref{main theorem}]
For any $X := X(R)$, there exists a unique $i$ such that $q_{i} \leq X < q_{i+1}$.  There moreover exists a $c=c(\alpha)> 0$ such that
\begin{equation}\label{eq:partial fraction bounds}
\frac{c}{q_{i}^{1+\eta}}<\big | \alpha - \frac{p_{i}}{q_{i}} \big | < \frac{1}{q_{i}q_{i+1}}.
\end{equation}
Hence
\[X < q_{i+1} < \frac{1}{c}\cdot q_{i}^{\eta},\]
which further implies that
\[X^{\frac{1}{\eta}} < q_{i+1}^{\frac{1}{\eta}} < c^{-\frac{1}{\eta}} \cdot q_{i}.\]
In other words, there exists a constant $C > 0$ such that
\[X^{\frac{1}{\eta}} < C \cdot q_{i}.\]
Let $ p=p_i$ and $q=q_i.$  By \eqref{eq:partial fraction bounds}, it follows that
\begin{equation}\label{eq:delta_bounds}
|\delta| = \big | \alpha - \frac{p_{i}}{q_{i}} \big | < \frac{1}{q_{i}q_{i+1}} < \frac{1}{q_{i}X} < C\cdot \frac{1}{X^{1+1/\eta}}.
\end{equation}
To ensure that $|\delta| < \epsilon/2$, we choose $X$ such that
\begin{equation}\label{eq:epsilon and X}
C \cdot X^{-(1+1/\eta)}\leq \frac{\epsilon}{2},
\end{equation}
namely, we subject $X$ to the restriction
 
\begin{equation}\label{eq:restrictions for asymptotic}
\epsilon^{-\frac{1}{1+1/\eta}} \ll X.
\end{equation}

To \textit{optimize} our error term, we seek a choice of $X$, subject to the restriction \eqref{eq:restrictions for asymptotic}, which minimizes the value of
\[O\left(\frac{R}{q}+\frac{1}{\epsilon q^{2}}+ \epsilon q R+1\right).\]
Note first that by \eqref{eq:delta_bounds} and \eqref{eq:epsilon and X},
\begin{equation}\label{eq:error bound 0}
\frac{1}{qX} \leq C \cdot \frac{1}{X^{1+1/\eta}} \leq \frac{\epsilon}{2},
\end{equation}
which in turn implies that
\[\frac{1}{q^{2}X^{2}} \le  \frac{\epsilon^{2}}{4},\]
and therefore that
\begin{equation}
\frac{1}{\epsilon q^{2}} \le  \frac{\epsilon X^{2}}{4}.\label{eq:error bound 1}
\end{equation}
Similarly, since \eqref{eq:error bound 0} implies $q^{-1} \leq \epsilon X/2,$ we find that
\begin{equation}\label{eq:error bound 2}
\frac{R}{q} = O\left(\epsilon X R\right).
\end{equation}
Next, since $q \le X$, it follows that
\begin{equation}\label{eq:error bound 3}
\epsilon q R = O\left(\epsilon X R\right),
\end{equation}
and finally it similarly follows from \eqref{eq:error bound 0} that
\begin{equation}\label{eq:error bound 4}
1  \le \frac{X}{q} = O\left( \epsilon X^2 \right).
\end{equation}
By \eqref{eq:error bound 1}, \eqref{eq:error bound 2}, \eqref{eq:error bound 3}, and \eqref{eq:error bound 4}, we have that
\[O\left(\frac{R}{q}+\frac{1}{\epsilon q^{2}}+ \epsilon q R+1\right) = O\left(\epsilon X^2 + \epsilon X R\right).\]
We thus choose the minimal possible value for $ X $, namely $ X \asymp \epsilon^{-\frac{1}{1+1/\eta}}$, which is moreover $ o(R) $ by the assumption that $ \epsilon R^{1+1/\eta} \to \infty$.  In particular,
\begin{equation}
\epsilon X^2  + \epsilon X R \ll \epsilon X R \ll  \epsilon^{\frac{1}{1+\eta}} R.\label{eq:FinalError}
\end{equation}
By \eqref{eq:Delta in q} and \eqref{eq:FinalError}, we conclude that
\begin{equation}\label{eq:Delta_irrational_asymptotic}
\Delta_{\alpha}(\epsilon,R) = \frac{\epsilon R^{2}}{1+\alpha^{2}}+O\left(\epsilon^{\frac{1}{1+\eta}} R\right).
\end{equation}
Theorem \ref{main theorem} now follows directly from \eqref{eq:Delta_irrational_asymptotic} and Lemma \ref{S and Delta}.

\end{proof}

\section{Sectors about Rational Slopes}\label{sec:rational_slopes}
In this section we provide a proof of Theorem \ref{rational case thm}, namely a count for $S_{\alpha}(\epsilon,R)$ when $\alpha \in \Q$.  The proof proceeds similarly to that of Theorem \ref{main theorem}, upon setting $\delta = 0$:
\begin{proof}[Proof of Theorem \ref{rational case thm}]
Recall that $\alpha = p/q$, where $(p,q)=1$.  Note that
\begin{align*}
\Delta_{\alpha}(\epsilon,R) &= \#\left\{(m,n) \in \Z^{2}: m(p/q-\epsilon)< n < m(p/q+\epsilon): 1 \leq m \leq \frac{qR}{\sqrt{p^{2}+q^{2}}}\right\}\\
&= \#\left\{(m,n) \in \Z^{2}: \left|nq -mp\right|<mq\epsilon: 1 \leq m \leq \frac{qR}{\sqrt{p^{2}+q^{2}}}\right\}.
\end{align*}

Let $d = nq -mp$, and note that $|d|< mq\epsilon$ implies
\[|d| \leq \frac{\epsilon q^{2} R}{\sqrt{p^{2}+q^{2}}},\]
as well as that
\[\frac{|d|}{q\epsilon} < m.\]
Partitioning with respect to $d$, we then write
\begin{align*}
\Delta_{\alpha}(\epsilon,R)&= \sum_{|d|\leq \frac{\epsilon q^{2}R}{\sqrt{p^{2}+q^{2}}}} \#\left\{(m,n) \in \Z^{2}: nq -mp=d,  \frac{|d|}{q\epsilon} < m \leq \frac{qR}{\sqrt{p^{2}+q^{2}}}\right\}\\
&= \sum_{|d|\leq \frac{\epsilon q^{2} R}{\sqrt{p^{2}+q^{2}}}}\sum_{\substack{\frac{|d|}{q\epsilon}<m \leq \frac{Rq}{\sqrt{p^{2}+q^{2}}} \\ m \equiv -d \bar{p} \;(q)}}1,
\end{align*}
where $ \bar{p} $ denotes the inverse of $ p $ modulo $ q $  (in particular, if $\epsilon q^{2} R/\sqrt{p^{2}+q^{2}}< 1$, then the only contribution to $\Delta_{\alpha}(\epsilon,R)$ comes from the term $d=0$, i.e. points $(m,n) \in \Z^{2}$ lying precisely on the line $y = \alpha x$).   Upon setting
\[A:= \frac{R}{\sqrt{p^{2}+q^{2}}} \quad \textnormal{ and } \quad B:=\epsilon q^{2},\]
and recalling that
\[\sum_{|d|\leq x}|d| = \lfloor x \rfloor \left(\lfloor x \rfloor + 1\right),\]
it follows that
\begin{align*}
&\Delta_{\alpha}(\epsilon,R)=\sum_{|d|\leq AB}\left(A-\frac{|d|}{B}+O(1)\right)\\
&=A\left(1+2\left\lfloor AB\right\rfloor\right)-\frac{1}{B}\left\lfloor AB \right\rfloor \left( \left\lfloor AB\right\rfloor+1 \right) +O(1+AB).
\end{align*}
We furthermore note that
\begin{align*}
&A\left(1+2\left\lfloor AB\right \rfloor\right) = 2A^{2}B+\left(1-2\left\{AB\right\}\right)A
\end{align*}
and similarly that
\begin{align*}
&\frac{1}{B}\left\lfloor AB \right\rfloor \left( \left\lfloor AB \right\rfloor+1\right)=\left( A -\frac{\left\{AB\right\}}{B}\right)\left(1+AB-\left\{AB\right\}\right)\\
&=A^{2}B+A\left(1-2\left\{AB\right\}\right) -\frac{\{AB\}}{B}\left(1-\{AB\}\right).
\end{align*}
Combining the above expressions we see that
\begin{align*}
\Delta_{\alpha}(\epsilon,R)&=A^2B + \frac{1}{B}\{AB\}\left(1-\{AB\}\right)+O(1+AB)\\
&=
\frac{\epsilon q^{2}R^{2}}{p^{2}+q^{2}} +\frac{1}{q^{2}\epsilon}\left\{\frac{\epsilon q^2 R}{\sqrt{p^{2}+q^{2}}}\right\}\left(1-\left\{\frac{\epsilon q^2 R}{\sqrt{p^{2}+q^{2}}}\right\}\right) + O(1+\epsilon R),
\end{align*}
and the desired result now follows from Lemma \ref{S and Delta}.
\end{proof}

\section{Very Quickly Shrinking Sectors}\label{sec:very_quick_sectors}
Finally, in this section we provide a proof of Proposition \ref{quickly shrinking}, namely that when $\alpha \in \R$ is irrational of finite type $\eta$ and $\epsilon = o(R^{-1-\eta})$, we find that $S_{\alpha}(\epsilon, R)=0$ for sufficiently large $R$:
\begin{proof}[Proof of Proposition \ref{quickly shrinking}]
Since $ \alpha $ is of finite type $ \eta$, there exists a constant $c=c(\alpha)>0$ such that for all $(p,q) \in \Z \times \Z_{>0}$,

\[\left|\alpha - \frac{p}{q}\right|> \frac{c}{q^{1+\eta}}.\]

Take $R_{0}$ sufficiently large such that for any $ R>R_0 $ we have
 \[\epsilon < \frac{c}{R^{1+\eta}}.\]
Then for any $R > R_{0}$, and any $(p,q) \in \Z \times \Z_{>0}$ with $0 < q \le  R$, we find that
\[\left|\alpha - \frac{p}{q}\right|> \frac{c}{q^{1+\eta}} \ge \frac{c}{R^{1+\eta}}>\epsilon_.\]

It follows that for all $R > R_{0}$ we have $S_{\alpha}(\epsilon,R) = 0$, as desired.
\end{proof}

\end{document}